\documentclass[12pt]{amsart}

\usepackage{amssymb, amscd, txfonts}
\usepackage{graphicx}
\usepackage[all]{xy} 

\numberwithin{equation}{section}

\newtheorem{theorem}{Theorem}[section]

\newtheorem{lemma}[theorem]{Lemma}

\theoremstyle{definition}

\theoremstyle{remark}
\newtheorem{remark}[theorem]{Remark}

\newcommand{\R}{\mathbb{R}}
\newcommand{\C}{\mathbb{C}}

\begin{document}

%%%%%%% Title %%%%%%%%%%%%%%%%%%%%%%%%
\title[]{Square-integrability of holomorphic differential forms on locally symmetric varieties}
\author[]{Shouhei Ma}
%\thanks{} 
\address{Department~of~Mathematics, Institute~of~Science~Tokyo, Tokyo 152-8551, Japan}
\email{ma@math.titech.ac.jp}
%\subjclass[2020]{11F75}
%\keywords{} 
%\dedicatory{}

\begin{abstract}
We prove that every holomorphic differential form of non-top degree 
on a locally symmetric variety with Baily-Borel boundary of codimension $>1$ 
is square-integrable with respect to the invariant metric. 
This generalizes a theorem of Weissauer in the Siegel modular case. 
\end{abstract}

\maketitle

\section{Introduction}

Let $\mathcal{D}$ be a Hermitian symmetric domain and 
$\Gamma$ be an arithmetic subgroup of ${\rm Aut}(\mathcal{D})$, 
which we may assume to be neat without loss of generality for our purpose. 
The quotient $X=\mathcal{D}/\Gamma$ has the structure of a smooth quasi-projective variety. 
We assume that the boundary of the Baily-Borel compactification of $X$ 
has codimension $>1$. 
The invariant Hermitian metric on $\mathcal{D}$ induces 
a natural complete K\"ahler metric on $X$. 
In this note we prove the following. 

\begin{theorem}\label{main}
Every holomorphic differential form of degree $< \dim X$ on $X$ 
is square-integrable with respect to the invariant metric. 
\end{theorem}

In the case of Siegel modular varieties, 
this was proved by Weissauer (\cite{We1} Satz 7) 
as a consequence of his classification of square-integrable Siegel modular forms. 
In other (higher-dimensional) cases, Theorem \ref{main} seems to be new. 
%for example, in the orthogonal case, it is not covered by \cite{Ma2} Chapter 10. 

Experience shows, at least in typical cases, that 
holomorphic differential forms of non-top degree correspond to 
vector-valued modular forms whose weight is so ``small'' that 
the automorphic representations they generate are no longer holomorphic discrete series. 
Theorem \ref{main} can be interpreted as saying that 
modular forms of such a small weight (of specific type) are always square-integrable. 
It ensures that the theory of discrete automorphic representations can be used to 
fully capture holomorphic differential forms. 
By contrast, in the case of top degree, 
square-integrability is equivalent to cuspidality (Remark \ref{rmk2}), 
and canonical forms with this property generate holomorphic discrete series. 

The proof of Theorem \ref{main} does not involve estimate of the $L^2$-norm, however. 
Instead, by making use of the mixed Hodge structure and the $L^2$-cohomology of $X$, 
Theorem \ref{main} is derived as a simple consequence of the Pommerening extension theorem \cite{Po}. 
%In a sense, Theorem \ref{main} could have been noticed earlier, but we could not find it in the literatures. 

I would like to thank Shuji Horinaga for asking this question in the orthogonal case.

\section{Proof}

\subsection{$L^2$-cohomology}\label{ssec: L2}

We begin with a general remark on $L^2$-cohomology. 
Let $X$ be a complete K\"ahler manifold whose $L^2$-cohomology 
$H^{k}_{(2)}(X)=H^{k}_{(2)}(X, {\C})$ in degree $k$ is finite-dimensional. 
Then, via harmonic forms, we have the Hodge decomposition (\cite{SZ} Theorem 1.2) 
\begin{equation}\label{eqn: Hodge}
H^{k}_{(2)}(X) = \bigoplus_{p+q=k} H^{p,q}_{(2)}(X), 
\end{equation}
where $H^{p,q}_{(2)}(X)$ is the $L^2$-$\bar{\partial}$-cohomology in degree $(p, q)$. 
This decomposition endows $H^{k}_{(2)}(X)$ with a real Hodge structure. 

On the other hand, we write 
$H^0_{(2)}(\Omega_{X}^{k})$ for the space of holomorphic $k$-forms on $X$ 
which are square-integrable with respect to the given metric. 
The following is well-known. 

\begin{lemma}\label{lem1}
We have $H^{k,0}_{(2)}(X)=H^0_{(2)}(\Omega_{X}^{k})$. 
\end{lemma}

\begin{proof}
For degree reason, there is no nonzero $\bar{\partial}$-exact form in degree $(k, 0)$, 
so $H^{k,0}_{(2)}(X)$ is the space of square-integrable $(k, 0)$-forms $\omega$ with $\bar{\partial}\omega=0$. 
The condition $\bar{\partial}\omega=0$ is just the holomorphicity of $\omega$. 
\end{proof}

We can also see Lemma \ref{lem1} via harmonic forms, 
noticing that a $(k, 0)$-form is harmonic if and only if it is holomorphic 
(as the adjoint of $\bar{\partial}$ is zero for degree reason).

\subsection{Proof of Theorem \ref{main}}

We go back to our locally symmetric variety $X=\mathcal{D}/\Gamma$ endowed with the invariant metric. 
Clearly this metric is complete. 
Moreover, $H^{k}_{(2)}(X)$ is isomorphic to the intersection cohomology of 
the Baily-Borel compactification $X^{bb}$ (\cite{Lo}, \cite{SS}), 
the latter being finite-dimensional by the projectivity of $X^{bb}$. 
Hence $X$ satisfies the assumption in \S \ref{ssec: L2}. 

We write $n=\dim X$. 
What has to be shown is $H^0_{(2)}(\Omega_{X}^{k})=H^0(\Omega_{X}^{k})$ for $k<n$. 
For the moment we allow $k\leq n$. 
By Lemma \ref{lem1}, we can identify $H^0_{(2)}(\Omega_{X}^{k})$ with 
the $(k, 0)$-component of $H^k_{(2)}(X)$. 
We consider the natural map 
\begin{equation}\label{eqn: L2 to ordinary}
H^{k}_{(2)}(X) \to H^k(X) 
\end{equation}
from the $L^2$-cohomology to the ordinary cohomology. 
Let $(F^{\bullet}, W_{\bullet})$ be the Hodge and weight filtrations in the mixed Hodge structure on $H^k(X)$. 
By Harris-Zucker (\cite{HZ} Chapter 5), 
the map \eqref{eqn: L2 to ordinary} is a morphism of mixed Hodge structures, 
and the image is $W_kH^k(X)$. 
(See also \cite{We2} Folgerung 4.3 for the latter.) 
Thus \eqref{eqn: L2 to ordinary} gives a surjective morphism 
$H^{k}_{(2)}(X) \twoheadrightarrow W_kH^k(X)$ 
of pure ${\R}$-Hodge structures. 
By the strictness of morphisms of Hodge structures, 
we obtain the surjective map 
\begin{equation}\label{eqn: L2 to ordinary Fk}
H^0_{(2)}(\Omega_{X}^{k}) = H^{k,0}_{(2)}(X) \twoheadrightarrow F^kW_kH^k(X). 
\end{equation}
For the proof of Theorem \ref{main} 
it is enough to know this surjectivity, 
but it will be more informative to know the following. 

\begin{lemma}\label{lem2}
The map \eqref{eqn: L2 to ordinary Fk} is an isomorphism 
and extends to an isomorphism $H^0(\Omega_{X}^{k}) \to F^kH^k(X)$. 
\end{lemma}

\begin{proof}
We take a smooth projective toroidal compactification $X\hookrightarrow \overline{X}$ 
with simple normal crossing boundary divisor $D$ (cf.~\cite{Ma1} Appendix). 
By mixed Hodge theory, we have a natural isomorphism 
$F^kH^k(X)\simeq H^0(\Omega_{\overline{X}}^{k}(\log D))$.  
Note that $\Omega_{\overline{X}}^{k}(\log D)$ is the canonical extension of $\Omega_{X}^{k}$ (\cite{Mu} Proposition 3.4). 
By the Koecher principle, we have 
$H^0(\Omega_{X}^{k}) = H^0(\Omega_{\overline{X}}^{k}(\log D))$. 
Therefore we obtain the commutative diagram 
\begin{equation*}\label{cd}
\xymatrix{
H^0_{(2)}(\Omega_{X}^{k}) \ar@{->>}[r] \ar@{^{(}->}[d] & F^kW_kH^k(X) \ar@{^{(}->}[d] \\ 
H^0(\Omega_{X}^{k}) \ar[r]^{\simeq} & F^kH^k(X). 
}
\end{equation*}
This shows that 
$H^0_{(2)}(\Omega_{X}^{k}) \to F^kW_kH^k(X)$ 
is injective. 
\end{proof}

The proof of Theorem \ref{main} can be completed as follows. 
Let now $k<n$. 
The Pommerening extension theorem \cite{Po} says that 
 $H^0(\Omega_{X}^{k}) = H^0(\Omega_{\overline{X}}^{k})$, 
which means $F^kH^k(X)=F^kW_kH^k(X)$ 
(see \cite{Ma1} Proposition 5.1). 
By Lemma \ref{lem2}, 
this implies $H^0_{(2)}(\Omega^{k}_{X}) = H^0(\Omega^{k}_{X})$. 
This proves Theorem \ref{main}. 

\begin{remark}
The above argument shows that 
Theorem \ref{main} is actually equivalent to the Pommerening extension theorem. 
In the Siegel modular case, this means that 
Weissauer's square-integrability theorem (\cite{We1} Satz 7) 
is equivalent to the extension theorem of Freitag and Pommerening \cite{FP}. 
\end{remark}

\begin{remark}\label{rmk2}
In the case $k=n$, 
Lemma \ref{lem2} says that we have 
\begin{equation*}
H^0_{(2)}(K_X) = F^nW_nH^n(X) = H^0(K_{\overline{X}}) 
\end{equation*}
inside $H^0(K_X) = F^nH^n(X) = H^0(K_{\overline{X}}(D))$. 
Since 
$H^0(K_{\overline{X}}) \subset H^0(K_{\overline{X}}(D))$ 
is the subspace of cusp forms, 
we recover the well-known fact that 
a canonical form (or a modular form of canonical weight) on $X$ is 
square-integrable if and only if it is a cusp form. 
%In some specific cases, 
%this has been verified by direct calculation (\cite{Fr} Satz III.2.6, \cite{Ma2} Theorem 10.1.(2)). 
Since the canonical weight belongs to the discrete series, 
this would also follow from the result of Wallach \cite{Wa}. 
\end{remark}

%%%%%%% Reference %%%%%%%%%%%%%%%%%%%%%%%%%%%%%

\end{document}